\documentclass[a4paper,12pt]{article}
\usepackage[utf8]{inputenc}
\usepackage{amsmath,amssymb,amsthm}
\usepackage{tikz-cd}
\usepackage{xcolor}
\usepackage{mathtools}
\usepackage{ulem}
\usepackage{setspace}
\usepackage{hyperref}
\newcommand{\mz}{\mathbb{Z}}
\newcommand{\mc}{\mathbb{C}}
\newcommand{\mr}{\mathbb{R}}
\newcommand{\mrz}{\mathbb{R}/\mathbb{Z}}

\newcommand{\hhat}{\hat{H}}

\newcommand{\mcA}{\mathcal{A}}

\newcommand{\man}{\textbf{Man}}
\newcommand{\ab}{\textbf{Ab}}

\newtheorem{proposition}{Proposition}
\newtheorem{corollary}[proposition]{Corollary}

\newtheorem{definition}[proposition]{Definition}

\newtheorem{question}{Question}

\title{Uniqueness of differential characters and differential K-theory via homological algebra}
\author{Ishan Mata}
\date{}

\begin{document}
\maketitle
%\doublespacing

\begin{abstract}
In \textit{Proc Math Sci 129, 70(219)}, Rakesh Pawar considers and solves a certain diagram extension problem. In this note, we observe that the existence and uniqueness of differential characters (defined as objects which fit into a certain hexagon diagram) follow directly from Rakesh Pawar's results. This provides an alternate proof of a weaker version of J. Simons and D. Sullivan's results (\textit{Journal of Topology, 2008, 1:45–56}). Further, this approach directly shows that the hexagon diagram uniquely determines the differential K-theory groups upto an isomorphism.
\end{abstract}

\section{Introduction}
In the seminal work \cite{cs2}, J. Cheeger and J. Simons introduced and developed the theory of differential characters. For a fixed smooth manifold $M$, they defined the abelian group of differential characters $\hhat^k(M;\mrz)$ as{\footnote{Our convention of degrees differs from the one in original paper \cite{cs2} where this group is called $\hhat^{k-1}(M;\mrz)$.}} \begin{equation}                                                                                                                                                                                                                            \hhat^k(M;\mrz)=\{ f: Hom(Z_{k-1}(M),\mrz)| f \circ \delta \in \Omega^k\}.                                                                                                                                                                                                                                    \end{equation}
Cheeger and Simons show that there are short exact sequences :
\begin{equation}
 0 \to H^{k-1}(M;\mrz) \xrightarrow{j} \hhat^k(M;\mrz) \xrightarrow{curv} \Omega_0^k(M) \to 0,
\end{equation} and 
\begin{equation}
 0 \to \frac{\Omega^{k-1}(M)}{\Omega_0^{k-1}(M)} \xrightarrow{i} \hhat^k(M;\mrz) \xrightarrow{c} H^k(M;\mz) \to 0.
\end{equation}
Here $\Omega^k_0(M)$ stands for closed degree $k$ forms with integral periods. They further showed that the composition $H^{k-1}(M;\mrz) \xrightarrow{j} \hhat^k(M;\mrz) \xrightarrow{c} H^k(M;\mz) $ is $-B$ where $B$ is the Bockstein map; and the composition $\frac{\Omega^{k-1}(M)}{\Omega_0^{k-1}(M)} \xrightarrow{i} \hhat^k(M;\mrz) \xrightarrow{curv} \Omega_0^k(M)$ is the exterior derivative $d$. 

There are various other constructions of differential refinements of ordinary singular cohomology in the literature, see for example \cite{gajer,brylinski,lawsona,lawsonb,bks}. It is natural to ask whether these constructions are equivalent. In \cite{ssdc}, Simons and Sullivan provide an axiomatic characterization of differential cohomology. 
They show that the following hexagon diagram with exact diagonals \cite{ssdc} :  
\begin{center}
\setlength{\unitlength}{0.5cm}
\begin{picture}(24,16)
\put(5,0.5){$0$}
\put(20.5,0.5){$0$}
\put(5.5,1){\vector(1,1){1.8}}
\put(18.5,3){\vector(1,-1){1.8}}
\put(8,4){$\frac{\Omega^{k-1}(M)}{\Omega^{k-1}_{0}(M)}$}
\put(12.5,4){\vector(1,0){3.5}}
\put(17,4){$\Omega^{k}_{0}(M)$}
\put(13.5,4.5){\small{$d$}}
\put(2.5,6){\vector(1,1){1.8}}
\put(6.5,7){\vector(1,-1){1.8}}
\put(7.5,7){\small{$\beta$}}
\put(10.5,7){\small{$i$}}
\put(10,5.5){\vector(1,1){1.8}}
\put(14.5,7){\vector(1,-1){1.8}}
\put(15.0,7){\small{$curv$}}
\put(18.5,7){\small{$s$}}
\put(19,5){\vector(1,1){1.8}}
\put(22.5,7.5){\vector(1,-1){1.8}}
\put(4,8){$H^{k-1}(M;\mr)$}
\put(11,8){$\hhat^{k}(M;\mrz)$}
\put(20,8){$H^{k}(M;\mr)$}
\put(2.5,11){\vector(1,-1){1.8}}
\put(6.5,11){\small{$\alpha$}}
\put(6,9.5){\vector(1,1){1.8}}
\put(10.5,11){\vector(1,-1){1.8}}
\put(11.5,10.5){\small{$j$}}
\put(14,10.5){\small{$c$}}
\put(14.0,9.5){\vector(1,1){1.5}}
\put(18.5,11){\vector(1,-1){1.5}}
\put(19.5,10.5){\small{$r$}}
\put(22.5,9.5){\vector(1,1){1.5}}
\put(6,12){$H^{k-1}(M;\mrz)$}
\put(12.2,12.2){\vector(1,0){3.2}}
\put(16,12){$H^{k}(M;\mz)$}
\put(13.5,12.5){\small{$-B$}}
\put(5.5,15){\vector(1,-1){1.8}}
\put(18,13.5){\vector(1,1){1.8}}
\put(4.5,15.5){$0$}
\put(20,15.5){$0$}
\end{picture}
Differential cohomology hexagon diagram
\end{center}
uniquely characterizes the ordinary differential cohomology functor upto a natural equivalence (for a precise statement, see proposition \ref{ssdcth}).
The long exact sequence of the upper arrows is the long exact sequence in cohomology corresponding to the short exact sequence $ 0 \to \mz \to \mr \to \mrz \to 0 $. The map $\beta$ and $s$ are induced by the the de-Rham morphism. %In their landmark paper \cite{ssdc}, Simons and Sullivan showed that the functor $\hhat^k(M;\mrz)$ and the natural transformations $i,j,c,curv$ are determined uniquely upto an isomorphism compatible with the other maps in the hexagon diagram.
 \\
This axiomatization of differential cohomology is useful since it establishes that different constructions of differential refinements of ordinary singular cohomology are essentially equivalent.
For example, since the Deligne cohomology functor \cite{brylinski} - defined as hypercohomology of a certain complex- fits in the hexagon diagram \cite{gajer}, it follows that the Cheeger-Simons differential character functor is naturally equivalent to the Deligne cohomology functor via an equivalence compatible with the diagonal morphisms in the respective hexagon diagrams.

Just as ordinary differential cohomology admits a differential refinement, so do other generalised cohomology theories. Given a generalised cohomology theory represented by a spectrum, Hopkins and Singer gave a  prescription \cite{hopkinssinger} for constructing its differential refinement. In particular one can construct a differential version of K-theory using their prescription. In \cite{ssdk}, Simons and Sullivan construct another model of differential K-theory for compact manifolds in terms of structured vector bundles. They show that their model of differential K-theory fits in the following hexagon diagram with exact diagonals : 
\begin{center}
\setlength{\unitlength}{0.5cm}
\begin{picture}(24,16)
\put(4.5,0.5){$0$}
\put(20.5,0.5){$0$}
\put(5.5,1){\vector(1,1){1.8}}
\put(18.5,3){\vector(1,-1){1.8}}
\put(8,4){$\frac{\Omega^{odd}(M)}{\Omega_{GL}(M)}$}
\put(12,4.2){\vector(1,0){3.5}}
\put(17,4){$\Omega_{BGL}(M)$}
\put(13.5,4.5){\small{$d$}}
\put(2.5,6){\vector(1,1){1.8}}
\put(6.5,7){\vector(1,-1){1.8}}
\put(7.2,6.6){\small{$\zeta$}}
\put(10.5,7){\small{$i$}}
\put(10.5,6){\vector(1,1){1.8}}
\put(14.5,7){\vector(1,-1){1.8}}
\put(15.0,7){\small{$curv$}}
\put(18.5,7){\small{$s$}}
\put(18.5,6){\vector(1,1){1.8}}
\put(22.5,7){\vector(1,-1){1.8}}
\put(4,8){$H^{odd}(M;\mc)$}
\put(12,8){$\hat{K}(M)$}
\put(20,8){$H^{even}(M;\mc)$}
\put(2.5,11){\vector(1,-1){1.8}}
\put(6.5,10.5){\small{$\chi$}}
\put(6,9.5){\vector(1,1){1.8}}
\put(10.5,11){\vector(1,-1){1.8}}
\put(11.5,10.5){\small{$j$}}
\put(14,10){\small{$\delta$}}
\put(14.0,9.5){\vector(1,1){1.8}}
\put(18.5,11){\vector(1,-1){1.8}}
\put(19.5,10.5){\small{$r$}}
\put(22.5,10){\vector(1,1){1.8}}
\put(7,12){$K(\mc/\mz)(M)$}
\put(12.3,12.2){\vector(1,0){3}}
\put(16,12){$K(M)$}
\put(13,12.5){\small{$-B$}}
\put(5.5,15){\vector(1,-1){1.8}}
\put(18,13){\vector(1,1){1.8}}
\put(5,15){$0$}
\put(20,15){$0$}
\end{picture}
\\Differential K-theory hexagon diagram
\end{center}
They ask whether, like in the case of differential characters, 

\begin{question}\label{ssquestion}
Does the above hexagon diagram determine the differential K-theory functor (from the category of compact manifolds to the category of abelian groups) upto a natural equivalence compatible with the respective diagonal morphisms ?
\end{question}
In \cite{rp}, Rakesh Pawar finds necessary and sufficient conditions for the diagram 
% https://tikzcd.yichuanshen.de/#N4Igdg9gJgpgziAXAbVABwnAlgFyxMJZARgBoAGAXVJADcBDAGwFcYkRyQBfU9TXfIRRli1Ok1bsACt14gM2PASIAmUqJoMWbRCACisvosFEAzOrFbJugEqH5-JUOTmVliTpAAxewoHKUc1N3bXYARV9HE2FSN00PdgAJSOMAklJg+NDdAGUU-2c1TPFskABxfKcicgss6w5K6OQa4qtPTh4jAqIyABYQ+o65PyrA0n669sa08ypJ9iGu0eRe2pLB6edV1oTdDrEYKABzeCJQADMAJwgAWyQAThocCCRiTpAr29enl8QVd8+dz+PyQpgB1yB5hAz1evXBX0QNWhvzeckB32RSAAbPCgViQYgAOy4pCEgnEVEXCGkgkADhJiGISJhRIZtIJAFYGRzyVy0dTEFCWXD+QjVpjEHyqQieRLiP9RUDiELfmDKFwgA
\begin{center}
\begin{tikzcd}\label{gapdiagram}
            & 0 \arrow[d]           &             & 0 \arrow[d]           &   \\
0 \arrow[r] & P \arrow[r] \arrow[d] & E \arrow[r] & R \arrow[r] \arrow[d] & 0 \\
            & H \arrow[d]           &             & F \arrow[d]           &   \\
0 \arrow[r] & S \arrow[d] \arrow[r] & G \arrow[r] & Q \arrow[r] \arrow[d] & 0 \\
            & 0                     &             & 0                     &  
\end{tikzcd}
\\ Diagram 1
\end{center}

with short exact rows and columns to extend to \\ 
% https://tikzcd.yichuanshen.de/#N4Igdg9gJgpgziAXAbVABwnAlgFyxMJZARgBoAGAXVJADcBDAGwFcYkRyQBfU9TXfIRRli1Ok1bsACt14gM2PASIAmUqJoMWbRCACisvosFEAzOrFbJugEqH5-JUOTmVliTpAAxewoHKUc1N3bXYARV9HE2FSN00PdgAJSOMAklJg+NDdAGUU-2c1TPFskABxfKcicgss6w5K6OQa4qtPTh4jAqIyABYQ+o65PyrA0n669sa08ypJ9iGu0eRe2pLB6edV1oTdRYdUwooBqc6D7pRVuPXT4ai0muu2hc3VcZOXs5GmtSfdkAAGtwxDAoABzeBEUAAMwAThAALZIACcNBwECQxDOcMRmLRGMQKmx8KRhPxSFMxNxiHMIHRmN6VNJNTpBKxchxpLIrKQADYmXzyYgAOwCkVC4jsmEkpDCoUADjFxBZ9JFYvlQoArGLNRLtRyZTShYyDdTVjzEPrpdTdRbiETTVzaarKY7MRqLbyaIwsGBPFB6HAABag+ycwUWlQsn1+9gB4OhsVR43e33+wMhqBhw3m1XEOUgGPphNZpVeyOp2O6eOZ7PUtSR6NpuMZxNuwkqgkeostkt10kevOowvN6ut0uULhAA
\begin{center}
% https://tikzcd.yichuanshen.de/#N4Igdg9gJgpgziAXAbVABwnAlgFyxMJZARgBoAGAXVJADcBDAGwFcYkRyQBfU9TXfIRQAmCtTpNW7Tjz7Y8BIgGYxNBizaIO3XiAzzBRcqWLj1UrTN36BilGVNrJmkAAUdc20OSjHEjewAoh56-AreKn7mLgBKITbhRAAsJmbO0vFhhijGwmkBlpkGdiSkeU4FIAASRV5EouX+FiAAGrWJKCqN0ewAYu3ZyCnd6YWyocXexkr5zVaeHaUzFc0AygMlostNLgDiGxGk2z1aAIoHyUezLvMTdfakSdcZ4wmDok8rNxedj89j4hgUAA5vAiKAAGYAJwgAFskMYQDgIEgUjskGBmIxGDRGPQAEYwRiuLJ2ECMGAQnAhaFwpBkJEoxAAVi+GKxOPJBKJJMm7ApVJAuKwYBcUHocAAFkCaTD4YhRIykAA2NmITHY3Hc4mkoTkynU8a0+UqJWINEnDWcvGEnV8rQCw26Y2omjIpCs9FaAA63sxQq5tt5931gqNctdZoAnGqQL7YcwAzaebr+QbZXSWW6maqvVatUHUw70+HM573YhiIiTiAAFZJ7XBjqh6nC0XscVSmWl+W5isAdjV+cDKftLYzvez9IZlo5BdHIcdE6QAA4p4gY3m5yO7Yv022xRLpVBlxv18QLaNh8nd82lz2kJuK1XY1gG4Wx0uDx2j93nRHK0RZ8Z1GEB4XnW9snHb8tE7Y9TxfM1iE9GtCAgpsoK-ckRUPLsTwfSsGWfXNZ01HcMLJe9-0zYgiKZYhFVI61GyLaDsPbWDf3w6j5WIU1n0vAprxYz8Sx4+k0WfNctzIm8KL1KjIQAi9zxQq9tzk1jFJAF1K3LeiSPU2SRL3MNxL089NyY9CtP3djcPggjiD7ejBxk5iP1Mp0lJolykGEasjI8hc7xLSguCAA
\begin{tikzcd}
                    & 0 \arrow[d]                         & 0 \arrow[d, dashed]                             & 0 \arrow[d]                   &   \\
0 \arrow[r]         & P \arrow[r, "\nu"] \arrow[d, "\mu"] & E \arrow[r] \arrow[d, "j", dashed]              & R \arrow[r] \arrow[d]         & 0 \\
0 \arrow[r, dashed] & H \arrow[d] \arrow[r, "i", dashed]  & X \arrow[r, "m", dashed] \arrow[d, "n", dashed] & F \arrow[d] \arrow[r, dashed] & 0 \\
0 \arrow[r]         & S \arrow[d] \arrow[r]               & G \arrow[r] \arrow[d, dashed]                   & Q \arrow[r] \arrow[d]         & 0 \\
                    & 0                                   & 0                                               & 0                             &  
\end{tikzcd}\\Diagram 2
\end{center}
with short exact rows and columns. He further gives conditions for uniqueness of such extensions.

In this modest note, we wish to highlight that the results of Rakesh Pawar directly imply (see proposition \ref{directdc}) the existence and uniqueness of differential character groups $\hhat^k(M)$(defined as objects which fit into the hexagon diagram with exact diagonals). Uniqueness of the functor $\hhat^k(-;\mrz)$ is a stronger result, for which we do not have a complete proof. However, we state a condition \ref{dcconditions} which implies the full Simons-Sullivan result.

Similarly we note in proposition \ref{directdk} that for any compact manifold $M$, the differential K-theory groups are uniquely determined upto an isomorphism compatible with the respective diagonal maps, thereby partially answering a question \ref{ssquestion} of Simons and Sullivan. We give necessary and sufficient conditions \ref{dkconditions} for an affirmative answer to the Simons-Sullivan question in full generality.

This note is organised as follows. In section \ref{rpsummary} we summarize the theorems of the article \cite{rp} that we need for our purposes. In \ref{mainsection} we state the uniqueness results for differential characters and differential K-theory. These uniqueness results are a direct corollary of Rakesh Pawar's purely homological algebraic results involving no topology or geometry. Interestingly, for this reason, this approach may potentially admit an adaptation for axiomatising other generalised differential cohomology theories.

Axiomatic characterization of generalised differential cohomology theories with fiber integration, and their uniqueness has been discussed in \cite{bunke}.
\section{Statement of Rakesh Pawar's results}\label{rpsummary}

In this section, we summarize the results of \cite{rp} that we need for present purposes. Let us begin by recalling some standard preliminary definitions and results from homological algebra (see, for example, \cite{weibel,rotman}). If $\mcA$ is an abelian category with enough projectives, then for any two objects $P,Q$ in $\mcA$, one can consider the groups $Ext^n(Q,P)$ as the derived functor of the $Hom$ functor.

Alternatively, one can consider the group of Yoneda extensions of $P$ by $Q$ as follows. Consider the set of long exact sequences $\zeta : 0 \to P \to X_n \to \cdots \to X_1 \to Q \to 0$. If $\zeta' : 0 \to P \to X'_n \to \cdots \to X'_1 \to Q \to 0$ is another such extension, a map $f: \zeta \to \zeta'$ is a collection of maps $f_i : X_i \to X'_i$  such that the diagram\\ % https://tikzcd.yichuanshen.de/#N4Igdg9gJgpgziAXAbVABwnAlgFyxMJZABgBpiBdUkANwEMAbAVxiRGJAF9T1Nd9CKAIzkqtRizYBFLjxAZseAkQBMo6vWatEIABoB9Qt16KBRAMzrxWtgDp7sk-2UoALFc2SdBoY-l8lQWQAVg8JbRAABT8FZyCANjCbHQ5jf1MXElIhMU8I1LlYwKIRHI1w6RiAsxQ1MusvPQByQyqMoMt6vLsHNKKa5Hcuiu99Xz7qzNDh5Ki2uKJEmcbUsRgoAHN4IlAAMwAnCABbJDIQHAgkITSD46vqC6QVG8OTxEtzy8RXF7vvh6+wV+b1CnyQ8WBSAA7ACkAAOSGIOGwxAATkRqJRQmIiOxWOuclubyEIjBiCEz0JryuajJQnMuJRyJADCwYAikHZIGoAAsYHQoGxOawHnQsAwhQQRSy6AAjGAMSKTQQgNnYWB+IlPFGYhoRXatRHuOmk7o6A3jKl-WmPd6MsnMs2qqD6GSI0G2inlWZYF3RTgUThAA
\begin{tikzcd}
0 \arrow[r] & P \arrow[r] \arrow[d, phantom] \arrow[d, "id_P"] & X_n \arrow[d, "f_n"] \arrow[r] & ... \arrow[r] & X_1 \arrow[r] \arrow[d, "f_1"] & Q \arrow[r] \arrow[d, "id_Q"] & 0 \\
0 \arrow[r] & P \arrow[r]                                      & X'_n \arrow[r]                 & ... \arrow[r] & X_1 \arrow[r]                  & Q \arrow[r]                   & 0
\end{tikzcd} \\commutes. Define an equivalence relation $\zeta \sim \eta \iff \exists$ a finite zigzag chain $\zeta \to \alpha_1 \leftarrow \alpha_2 \to \alpha_3 \leftarrow \cdots \to \eta$. Quotient of the set of extensions considered above by this equivalence relation gives us the set of Yoneda extensions $Ext^n_{Yoneda}(Q,P)$. On this set, define addition as $\zeta+\zeta' = [0 \to P \to Y_{n} \to X'_{n-1} \oplus X_{n-1} \to \cdots \to X'_2\oplus X_2 \to Y_1 \to Q \to 0]$. Here $Y_1$ is the pullback $X_1 \times_Q X'_1$, and $Y_n$ is the quotient by a skew diagonal copy of $P$, of the pushout of $P \to X_n$ and $P \to X'_n$. The set $Ext^n_{Yoneda}(Q,P)$ becomes an abelian group under this operation. If the category $\mcA$ has enough projectives, then $Ext^n_{Yoneda}(Q,P)$ is isomorphic to $Ext^n(Q,P)$ (see, for example, section 3.4 of \cite{weibel}). Throughout this article we shall assume that the category $\mathcal{A}$ has enough projectives. 
\begin{proposition}\label{rpthme} (Rakesh Pawar) : Let $\mathcal{A}$ be an abelian category. Let Diagram 1 have exact rows and columns of objects of $\mathcal{A}$. The diagram 1 extends to diagram 2 with exact rows and columns if and only if the Baer sum of $[E]\cup[F]$ and $[H]\cup[G]$ is zero in $Ext^2(Q,P)$.
\end{proposition}
In \cite{rp}, this proposition is stated for small categories, however as noted in the remark 3.3 of \cite{rp}, the result holds good for general abelian categories.\\
Stating the uniqueness theorem requires a bit of background. First pull back the exact sequence $0 \to R \to F \to Q \to 0$ by the map $G \to Q$ to get : 
% https://tikzcd.yichuanshen.de/#N4Igdg9gJgpgziAXAbVABwnAlgFyxMJZABgBoBGAXVJADcBDAGwFcYkRiQBfU9TXfIRTkK1Ok1bsASt14gM2PASIAmUTQYs2iEADFZfRYKIBmdeK3sAigfn8lQ5ABZzmyTs49DA5cNLExN20QGS87I19kNQCNCWCATVsFH0czGIt3EABxJPtjFDJ0oPZPOWSHIhciuJLuMRgoAHN4IlAAMwAnCABbJDIQHAgkcjDOnuGaQaQVUa7exDUBocQTWfGVyeWnNfmADk2kAFYdo4PEADYTi7OAdiubs4BOe7PVuTH587OZ97nTpYmGWCWCgAH1QpQuEA
\begin{center}
\begin{tikzcd}
0 \arrow[r] & R \arrow[r] \arrow[d, "id_R"] & Y \arrow[r] \arrow[d] & G \arrow[r] \arrow[d] & 0 \\
0 \arrow[r] & R \arrow[r]                   & F \arrow[r]           & Q \arrow[r]           & 0
\end{tikzcd}
\end{center}
Applying the Snake lemma, Rakesh Pawar obtains 
% https://tikzcd.yichuanshen.de/#N4Igdg9gJgpgziAXAbVABwnAlgFyxMJZAZgBoBGAXVJADcBDAGwFcYkQAlEAX1PU1z5CKMgCZqdJq3YAxHnxAZseAkTLEJDFm0QgAivP7KhRURU1SdnQ4oErhyM+JpbpugJo2lg1SjMaXS3YAcS87ExRyUmdJbXYAZTDjX2QABmiLON1UpJ8HdIDYtxAc3iM8oijC1ytEsttkhzJUzOLShW97U1IWwKyS3K6-UgAWVqt28qGSUfH2SYaKlBGMvrbBiOQV6qDsjZSVqjWJ-YcoseP5ngkYKABzeCJQADMAJwgAWyQAdhocCCQADZ6m9PkC-gDECMQe8vlCIUhyDCwYgoiB-ojoQpQXCABwIxAATmRcMBBMJl10WCgAH06tjYUgKejIQBWElIVkE0QcxBmFmI9kMlHkdICxCpXlijGo3lomU84Vw-ky8jEXlc8XkRUvRnw8VC3UosjirFGuEmmVimrsak0rhy+WQ9VKpDSyHkYGuokE8i-ECMLBgKxwCCBqAgGgACxg9AjiDAzEYjD+9CwjHYkGDkZAcCjWGeODd3Eo3CAA
\begin{center}
\begin{tikzcd}
            &                               & 0 \arrow[d]                   & 0 \arrow[d]           &   \\
            &                               & R \arrow[d] \arrow[r, "id_R"] & R \arrow[d] \arrow[r] & 0 \\
0 \arrow[r] & S \arrow[r] \arrow[d, "id_S"] & Y \arrow[r] \arrow[d]         & F \arrow[r] \arrow[d] & 0 \\
0 \arrow[r] & S \arrow[r] \arrow[d]         & G \arrow[r] \arrow[d]         & Q \arrow[r] \arrow[d] & 0 \\
            & 0                             & 0                             & 0                     &  
\end{tikzcd}
\end{center}.
The injective maps $R \to Y$, and $S \to Y$ induce a map $R \oplus S$ induce a short exact sequence $0 \to R \oplus S \to Y \to Q \to 0$. Applying the functor $Hom(-,P)$, one obtains the long exact sequence
\begin{equation}\label{rpexact}
\cdots \to Hom(R \oplus S, P) \xrightarrow{\alpha} Ext^1(Q,P) \xrightarrow{\beta} Ext^1(Y,P) \to \cdots.  
\end{equation}
 
\begin{proposition}\label{rpthmu}
 (Rakesh Pawar) If the map $\alpha$ is surjective, then $[X] \in Ext^1(Y,P)$ is unique.
\end{proposition}
Thus if $X_1$ is another abelian group together with maps $i_1,j_1,m_1,n_1$, then $[X]=[X_1] \in Ext^1(Y,P)$. Equivalently there is an abelian group isomorphism $\phi : X \to X_1$ such that the diagram% https://tikzcd.yichuanshen.de/#N4Igdg9gJgpgziAXAbVABwnAlgFyxMJZARgBoAGAXVJADcBDAGwFcYkQAFEAX1PU1z5CKMsWp0mrdl179seAkXIVxDFm0QhyPPiAzyhS0mJpqpm7bL0CFw5ACYVpyRpAANHXMGKUAZicS6uwAmp7WBj4Oxqou7G4A+sRh+t52-iaB5iChVim2RAAs0c5BFsk2hihFVCVZluIwUADm8ESgAGYAThAAtkiOIDgQSORWXb0jNENIBWPdfYhFg8OIAKxzE2tTKwCcGwv+y0jE+8fbSABsp4gX54gA7Nf3dwAc18pHiGSZrlhQ8TJdOMFktpjdaq4ADqQtAACywYWBSFWd2eP3Yf3iOUo3CAA
\begin{center}
\begin{tikzcd}
0 \arrow[r] & P \arrow[r] \arrow[d, "id_P"] & X \arrow[r] \arrow[d, "\phi"] & Y \arrow[r] \arrow[d, "id_Y"] & 0 \\
0 \arrow[r] & P \arrow[r]                   & X_1 \arrow[r]                 & Y \arrow[r]                   & 0
\end{tikzcd} 
\end{center}
commutes. Since $Y$ is the pullback 
\begin{center}
 % https://tikzcd.yichuanshen.de/#N4Igdg9gJgpgziAXAbVABwnAlgFyxMJZABgBpiBdUkANwEMAbAVxiRAE0BeAMQAIAdfkzCwATnBg5gARQC+wQXgC28WbwDiIWaXSZc+QigCM5KrUYs23LTpAZseAkTJGz9Zq0QhN23Q4NEJq7U7pZe0lpmMFAA5vBEoABmohBKSCYgOBBIAMy+IMmpSABM1Fm5+YVpiGSZ2YhGlSnVteWIpZl0WAxskGCsshSyQA
\begin{tikzcd}
Y=F \underset{Q}{\times} G \arrow[r] \arrow[d] & F \arrow[d] \\
G \arrow[r]                                    & Q          
\end{tikzcd},
\end{center}
it follows that the morphism $\phi$ is compatible with $(m,m_1)$ and $(n,n_1)$ i.e. $m_1 \circ \phi =m$ and $n_1 \circ \phi =n$, and that $\phi \circ i|_P =i_1|_P$, $\phi \circ j|_P =j_1|_P$. (Here we are considering $P$ as a subgroup of $H$ via $\mu$, and of $E$ via $\nu$.) Alternatively we could say that $\phi \circ i \circ \mu =i_1 \circ \mu$, and $\phi \circ j \circ \nu = j_1 \circ \nu$. 
However, we need a stronger compatibility result for our purposes : $\phi \circ i =i_1$, and $\phi \circ j = j_1$. We obtain  this in the next section.
\section{Existence and uniqueness results for differential cohomology theories}\label{mainsection}
Let $(X_1,i_1,j_1,m_1,n_1)$ and $(X_2,i_2,j_2,m_2,n_2)$ be two extensions of Diagram 1. Let us say that an isomorphism $\phi : X_1 \to X_2$ is a \textit{compatible isomorphism} between these two extensions if $\phi \circ i_1 =i_2, \phi \circ j_1 =j_2, m_2 \circ \phi =m_1$, and $n_2 \circ \phi =n_1$. The following proposition gives a necessary and sufficient criterion for an extension (assuming it exists) to be determined upto such equivalence.
\begin{proposition}\label{finaluniqueness}
 Suppose $(X_1,i_1,j_1,m_1,n_1)$ is an extension of Diagram 1. Let $E_1=j_1(E) \subset X_1$, and $H_1=i_1(H) \subset X_1$.
 The the following are equivalent :
 \begin{enumerate}
  \item For any other extension $(X_2,i_2,j_2,m_2,n_2)$ of the diagram, there exists an isomorphism $\phi : X_1 \to X_2$ compatible with the two extensions.
  \item The map $Hom(R \oplus S, P) \xrightarrow{\alpha} Ext^1(Q,P)$ is surjective, and every homomorphism $\lambda : E_1+H_1 \to P$ which vanishes on $P_1 \subset E_1+H_1$ admits an extension $\Lambda : X_1 \to P$.

 \end{enumerate}
\end{proposition}\begin{proof} We shall show that (2) $\implies$ (1). The other direction follows by retracing the steps of the proof.\\
By Rakesh Pawar's result \ref{rpthmu}, there is an isomorphism $\phi : X_1 \to X_2$ such that $m_2 \circ \phi =m_1$, and $n_2 \circ \phi =n_1$. The strategy is to find a morphism $\eta : X_1 \to X_2$ such that the morphism $\phi'\equiv \phi + \eta$ is a compatible isomorphism i.e. $\phi' \circ i_1 =i_2, \phi' \circ j_1 =j_2, m_2 \circ \phi' =m_1$, and $n_2 \circ \phi' =n_1$.
\begin{center}
 % https://tikzcd.yichuanshen.de/#N4Igdg9gJgpgziAXAbVABwnAlgFyxMJZARgBpiBdUkANwEMAbAVxiRAAUQBfU9TXfIRQBmclVqMWbAKLdeIDNjwEiAVjHV6zVohAAlOXyWCiogEzitU3QA0A+sUML+yocgAspYZck6Q9sydFARUUMgtNXzYACSCXEzDSdx9tNgBlOONQ5FFkyNTdAHFMkLd1PIkCkABFEtciAAYNSusQBrqEklIGlNb2niNS027ev375YPqUdR78vo7sgDZmqzGFstIIlr8AMXWiZa3VtnHBqeQmo6jdU+cstzJVUZP9kVInubWBu6Hp9+ebq8LkkAW0gcsKsdAVxxDAoABzeBEUAAMwAThAALZIACc1BwECQDW+6KxRPxhMQZG2bAAOrSwEwnKTsVSKUgzCSMayzOyqcIuWSqU0QASiYLWSKxYh1DTdPTMUyJUhZdLFsrEMtRZTiAB2DVmKWU9XyFlILXS-Wm7lIXV8gAcGvtfLMxA1xAtlNUGtVlNEcpAWAczJtiH90uIngDmOD1AYdAARjAGOx4qFA2BsLAQ0LI3ziN7rULw5So1CQPS0AALLAgOOJ5Op+5sNFYeFVnA51m+pBl66BuyBetJlNpoQgVvtzsaqMRvtVGOBDUl22fNhgWMgeMjpu-Cdtjtd3t8u0BjdLousvPapCOy9IYjUiPuvn+8sAK0328bY5bB+n95UryN6IJygHAXOL4gfOrSfkOW4NqOza6AwMAogBqKhqeEZ3phQrOiBxA4jCXBAA
\begin{tikzcd}
            & 0 \arrow[d]                                         &  & 0 \arrow[d]                                                              &                                         & 0 \arrow[d]            &   \\
0 \arrow[r] & P \arrow[rr, "\nu"] \arrow[d, "\mu"]                &  & E \arrow[rr] \arrow[d, "j_1"'] \arrow[rdd, "j_2"]                        &                                         & R \arrow[r] \arrow[d]  & 0 \\
0 \arrow[r] & H \arrow[dd] \arrow[rr, "i_1"] \arrow[rrrd, "i_2"'] &  & X_1 \arrow[rr, "m_1" description] \arrow[rd, "\phi"'] \arrow[dd, "n_1"'] &                                         & F \arrow[r] \arrow[dd] & 0 \\
            &                                                     &  &                                                                          & X_2 \arrow[ru, "m_2"] \arrow[ld, "n_2"] &                        &   \\
0 \arrow[r] & S \arrow[d] \arrow[rr]                              &  & G \arrow[rr] \arrow[d]                                                   &                                         & Q \arrow[r] \arrow[d]  & 0 \\
            & 0                                                   &  & 0                                                                        &                                         & 0                      &  
\end{tikzcd}
\end{center}
For convenience, let us denote $P_1 \equiv i_1 \circ \mu (P) = j_1 \circ \nu (P) \subset X_1$, and $P_2 \equiv i_2 \circ \mu (P) = j_2 \circ \nu (P) \subset X_2$. Similarly let $E_1=j_1(E), E_2=j_2(E), H_1=i_1(H)$, and $H_2=i_2(H)$.\\
Now consider $\tilde{j} = j_2 - \phi \circ j_1 : E \to X_2$. Note that $m_2 \circ \tilde{j}=m_2 \circ j_2 - m_2 \circ \phi \circ j_1 = m_2 \circ j_2 -m_1 \circ j_1=0$, and $n_2 \circ \tilde{j}=n_2 \circ j_2 - n_2 \circ \phi \circ j_1 = m_2 \circ j_2 -n_1 \circ j_1=0$. Thus $\tilde{j}(E) \subset P_2$. Similarly $\tilde{i}(H) \subset P_2$. Also note that since $\phi \circ i_1|_{P} =i_2|_P$ and $\phi \circ j_1|_P =j_2|_P$ (by the discussion after proposition \ref{rpthmu}), we conclude that $\tilde{j}|_P=0=\tilde{i}|_P$. Hence, $\tilde{i}+\tilde{j} : E_1 + H_1 \to X_2$ is a well defined abelian group homomorphism taking values in $P_2$. Here we have identified $E$ with $E_1$, and $H$ with $H_1$, for notational simplicity we use the same notation for the maps $\tilde{i}$ and $\tilde{j}$. We thus have a commutative diagram : 
\begin{center}
 % https://tikzcd.yichuanshen.de/#N4Igdg9gJgpgziAXAbVABwnAlgFyxMJZARgBoAGAXVJADcBDAGwFcYkQBRAfWIGoAJHiAC+pdJlz5CKchWp0mrduRFiQGbHgJEATHJoMWbRCAAaQ0eM1SiZYvMNKTABS46R8mFADm8IqAAzACcIAFskMhAcCCRySxBgsNiaaKQdeMTwxFkomMQAZhpGLDBjEDgIYqgQA0UygB16vEZYYCxhXkbm1oArYRqQRnoAIxhGZwktaRAgrG8ACxxVQJCsvVykQsGSsqgIHBwvAccG+pgceg9hIA
\begin{tikzcd}
0 \arrow[r] & E_1+H_1 \arrow[r] \arrow[d, "\tilde{i}+\tilde{j}"'] & X_1 \arrow[ld, "\eta", dotted] \\
            & P_2                                                 &                               
\end{tikzcd}
\end{center}
By hypothesis, there exists an extension $\eta : X_1 \to X_2$. Let $\phi'\equiv \phi+\eta$. Then $\phi' \circ i_1 = \phi \circ i_1 + \eta \circ i_1 = \phi \circ i_1 + (i_2 - \phi \circ i_1) = i_2$, and $\phi' \circ j_1 = \phi \circ j_1 + \eta \circ j_1 = \phi \circ j_1 + (j_2 - \phi \circ j_1) = j_2$. Further since $\eta$ takes values in $P_2$, $m_2 \circ \eta =0$. Thus $m_2 \circ \phi'=m_2 \circ \phi + m_2 \circ \eta = m_2 \circ \phi = m_1$, and similarly $n_2 \circ \phi = n_1$. $\phi'$ is an isomorphism by the three lemma. Hence $\phi'$ is a required isomorphism.
\end{proof} 

\begin{corollary}\label{corpinj}
If $P$ is an injective object, there exists an extension of the diagram 1. Further for any two such extensions, there exists a compatible isomorphism between them.
\end{corollary}

Let $\man$ be the category of smooth manifolds and smooth maps between them, and let $\ab$ be the category of abelian groups and group homomorphisms.
\subsection{The case of differential characters}\label{dc}

As noted in the introduction, the Deligne cohomology groups defined as hypercohomology of a certain double complex \cite{brylinski} are isomorphic to the differential character groups $\hhat^k(M;\mrz)$ defined by Cheeger-Simons. Similarily, the de Rham-Federer currents \cite{lawsona,lawsonb} too provide a model of differential cohomology. In order to compare various models, it is important to axiomatically characterize ordinary differential cohomology. In \cite{ssdc} Simons and Sullivan define :
\begin{definition}\label{dcdef}
A functor $\hhat^k$ from $\man^{op}$ to $\ab$, together with natural transformations $i,j,c,curv$ is called a differential character functor if the following diagram in $Fun(\man^{op},\ab)$
\begin{center}
\setlength{\unitlength}{0.5cm}
\begin{picture}(24,16)
\put(4,0){$0$}
\put(21,0){$0$}
\put(5.5,1){\vector(1,1){1.7}}
\put(18.5,3){\vector(1,-1){1.7}}
\put(8,4){$\frac{\Omega^{k-1}(-)}{\Omega^{k-1}_{0}(-)}$}
\put(12.5,4){\vector(1,0){3.2}}
\put(17,4){$\Omega^{k}_{0}(-)$}
\put(13.5,4.2){\small{$d$}}
\put(2.5,6){\vector(1,1){1.7}}
\put(6.5,7){\vector(1,-1){1.7}}
\put(7.5,6.5){\small{$\beta$}}
\put(10.2,6.2){\small{$i$}}
\put(10,5.5){\vector(1,1){1.7}}
\put(14.5,7){\vector(1,-1){1.7}}
\put(15.2,6.5){\small{$curv$}}
\put(18.5,7){\small{$s$}}
\put(18.5,6){\vector(1,1){1.7}}
\put(22.5,7){\vector(1,-1){1.7}}
\put(4,8){$H^{k-1}(-;\mr)$}
\put(11,8){$\hhat^{k}(-;\mrz)$}
\put(20,8){$H^{k}(-;\mr)$}
\put(2.5,11){\vector(1,-1){1.7}}
\put(6,10){\small{$\alpha$}}
\put(6,9.5){\vector(1,1){1.7}}
\put(10.5,11){\vector(1,-1){1.7}}
\put(11.5,10.5){\small{$j$}}
\put(14,10.5){\small{$c$}}
\put(14.0,9.5){\vector(1,1){1.7}}
\put(18.5,11){\vector(1,-1){1.7}}
\put(19.5,10.5){\small{$r$}}
\put(22,10){\vector(1,1){1.7}}
\put(6,12){$H^{k-1}(-;\mrz)$}
\put(12.5,12.2){\vector(1,0){3}}
\put(16,12){$H^{k}(-;\mz)$}
\put(13.5,12.5){\small{$-B$}}
\put(5.5,15){\vector(1,-1){1.7}}
\put(18,13){\vector(1,1){1.7}}
\put(5,15){$0$}
\put(20,15){$0$}
\end{picture}
\end{center}
commutes and has exact diagonals.
\end{definition}
and prove
\begin{proposition}\label{ssdcth}
If $\hhat_1(-)$ and $\hhat_2(-)$ are two differential character functors from $\man^{op}$ to $\ab$ together with the natural transformations $(i_1,j_1,c_1,curv_1)$ and $(i_2,j_2,c_2,curv_2)$ (respectively), then there exists a unique natural equivalence $\psi : \hhat_1 \to \hhat_2$ which is compatible with the given maps i.e. $\psi \circ i_1 =i_2, \psi \circ j_1=j_2, c_2 \circ \psi =c_1$, and $curv_2 \circ \psi = curv_1$.
\end{proposition}
Here, we observe that the following proposition is a direct consequence of Rakesh Pawar's results \ref{rpthme},\ref{rpthmu} 
\begin{proposition}\label{directdc}
Let $M$ be a smooth manifold. Then there exists a group $G$ together with maps $i,j,c,curv$ such that the following diagram commutes and has short exact diagonals : 
\begin{center}
\setlength{\unitlength}{0.5cm}
\begin{picture}(24,16)
\put(4.5,0.5){$0$}
\put(20.5,0.5){$0$}
\put(5.5,1){\vector(1,1){1.7}}
\put(18.5,3){\vector(1,-1){1.7}}
\put(8,4){$\frac{\Omega^{k-1}(M)}{\Omega^{k-1}_{0}(M)}$}
\put(12.5,4){\vector(1,0){3.2}}
\put(17,4){$\Omega^{k}_{0}(M)$}
\put(13.5,4.5){\small{$d$}}
\put(2.5,6){\vector(1,1){1.7}}
\put(6.5,7){\vector(1,-1){1.7}}
\put(7.5,6.5){\small{$\beta$}}
\put(10.5,7){\small{$i$}}
\put(10.5,6){\vector(1,1){1.7}}
\put(14.5,7.4){\vector(1,-1){1.8}}
\put(15.0,7){\small{$curv$}}
\put(18.5,7){\small{$s$}}
\put(18.5,6){\vector(1,1){1.7}}
\put(22,7){\vector(1,-1){1.7}}
\put(4,8){$H^{k-1}(M;\mr)$}
\put(13,8){$G$}
\put(20,8){$H^{k}(M;\mr)$}
\put(2.5,11){\vector(1,-1){1.7}}
\put(6.5,10.5){\small{$\alpha$}}
\put(6.5,9.5){\vector(1,1){1.7}}
\put(10.5,11){\vector(1,-1){1.7}}
\put(11.5,10.5){\small{$j$}}
\put(14,10.5){\small{$c$}}
\put(14.0,9.5){\vector(1,1){1.7}}
\put(18.5,11){\vector(1,-1){1.7}}
\put(19.5,10.5){\small{$r$}}
\put(22,9.6){\vector(1,1){1.7}}
\put(6,12){$H^{k-1}(M;\mrz)$}
\put(13,12){\vector(1,0){2.5}}
\put(16,12){$H^{k}(M;\mz)$}
\put(13.5,12.5){\small{$-B$}}
\put(5.5,14.8){\vector(1,-1){1.7}}
\put(18,13){\vector(1,1){1.7}}
\put(5,15){$0$}
\put(20,15){$0$}
\end{picture}
\end{center}
Furthermore, if $G'$ is any other abelian group together with maps $i',j',c',curv'$ which make the diagram commute and have short exact diagonals, then there exists an isomorphism $\phi : G \to G'$ such that $\phi \circ i =i', \phi \circ j = j', c'\circ \phi =c$, and $curv' \circ \phi =curv$.
\begin{proof} First, note that the diagram above can be redrawn \cite{bb} as :
 % https://tikzcd.yichuanshen.de/#N4Igdg9gJgpgziAXAbVABwnAlgFyxMJZARgBpiBdUkANwEMAbAVxiRAB12AzAJzoGNgACQB6wANYBaYgF8AFAFlSnALZ0cACwBGW4ACUZAShnCxU2YuXs1mncABaRmSBml0mXPkIoyABiq0jCxsvi5uIBjYeAREvuQB9MysiCChru5RXrGkAEwJQcmpYRmeMShxAMz5SSHFER7R3iSkACzVwSlp4ZGlTTmt7YVdJY1EFQPUiR1F6fWZZcgtpFWTBbWzPaMoS3mrNZ11m1nb8XvTw3O9Y6T+Z0OHDcfI-beB+zPdjwv9lHdsnLwBMBOAB5FQwADmdDM0nkCmMwPYYMhdAA+r4YRZ4c4Nl8muNfm9plBQeCoZi4YYHvMmmRdkTCqIJLDLKp1NpdAYAPRs2y6RxU3E0ohkFYMtgAUQAHjg5EJUcyLAANADcvI5DiMVhsGoF1KuKH69KmhQA4vqtshxsa1ilSSj0SJxIpBZ9hYbloM2KJnQo1dZ2XY9UKDVbPX8UkIICo5ajnUrtYH+Vr1UGjC4AjAoBD4ERQLxo0h+iAcBAkL5ZgWVOXqKWkMQKpWeIXEA3a2XWy0my3iEsSx3iBXwlWkON+-WAKzd6utift+sAdgjIAAVnUR62l+PWwAOZf8Jg8Gjr5sz4h77cATmnSD7ddbADYb4-563r8PT-XL6+ckP85-EF-V8d2fNtt2ILcTTYLATx7Ld72Ib9xRSfhYLPO8BxAj8ewve9f2fOJwKnbCzzncCnxI+sH1fYj-x7JD7wouiZyA7cF1AshwK7SjW2LBDGx4zj7wrCgZCAA
 \begin{center}
\begin{tikzcd}
            & 0 \arrow[d]                                                               & 0 \arrow[d]                                                        & 0 \arrow[d]                                             &   \\
0 \arrow[r] & {\frac{H^{k-1}(M;\mathbb{R})}{H^{k-1}(M;\mathbb{Z})}} \arrow[r] \arrow[d] & \frac{\Omega^{k-1}(M)}{\Omega_0^{k-1}(M)} \arrow[r] \arrow[d, "i"] & d\Omega^{k-1}(M) \arrow[r] \arrow[d]                    & 0 \\
0 \arrow[r] & {H^{k-1}(M;\mathbb{R}/\mathbb{Z})} \arrow[r, "j"] \arrow[d]               & G \arrow[r, "curv"] \arrow[d, "c"]                                 & \Omega_0^k(M) \arrow[r] \arrow[d]                       & 0 \\
0 \arrow[r] & {Ext(H_{k-1}(X;\mathbb{Z}),\mathbb{Z})} \arrow[r] \arrow[d]               & H^k(M;\mathbb{Z}) \arrow[r] \arrow[d]                              & {Hom(H_k(X;\mathbb{Z}),\mathbb{Z})} \arrow[r] \arrow[d] & 0 \\
            & 0                                                                         & 0                                                                  & 0                                                       &  
\end{tikzcd}
\end{center}
The proposition follows from corollary \ref{corpinj} by noting that $P=\frac{H^{k-1}(M;\mr)}{H^{k-1}(M;\mz)}$ is divisible, and hence injective.
\end{proof}
\end{proposition}
Note that the proposition \ref{directdc} is weaker than the Simons-Sullivan theorem \ref{ssdcth}. The full Simons-Sullivan theorem is a statement about functors. We therefore consider the category of functors $Fun(\man^{op},\ab)$. This is an abelian category, having enough projectives and enough injectives \cite{mseinjectives,moprojectives}. Therefore by proposition \ref{finaluniqueness}, we have the following 
\begin{proposition}\label{dcconditions}
 The following are equivalent : 
 \begin{enumerate}
  \item The Simons-Sullivan hexagon diagram uniquely determines the functor $\hhat^n(M;\mrz)$ upto a compatible natural equivalence.
  \item The natural transformation $Hom(d\Omega^{n-1}(-) \oplus Ext(H_{n-1}(-;\mz),\mz),\frac{H^n(-;\mr)}{H^n(-;\mz)_\mr}) \xrightarrow{\alpha} Ext(Hom(H_n(-,\mz)),\frac{H^n(-;\mr)}{H^n(-;\mz)_\mr})$ is an epimorphism in $Fun(\man^{op},\ab)$, and every natural transformation from $Hom(H^{k-1}(-;\mathbb{R}/\mathbb{Z})+\frac{\Omega^{k-1}(-)}{\Omega_0^{k-1}(-)})$ to $\frac{H^n(-;\mr)}{H^n(-;\mz)_\mr}$ which vanishes on $\frac{H^n(-;\mr)}{H^n(-;\mz)_\mr}$ extends to a natural transformation on $\hhat^n(-;\mrz)$.
  %\item The morphisms in $Fun(\man^{op},\ab)$ :
%$Hom(d\Omega^{n-1}(-) \oplus Ext(H_{n-1}(-;\mz),\mz),\frac{H^n(-;\mr)}{H^n(-;\mz)_\mr}) \xrightarrow{\alpha} Ext(Hom(H_n(-,\mz)),\frac{H^n(-;\mr)}{H^n(-;\mz)_\mr})$ and  $Hom(\hhat^n(-;\mrz),\frac{H^n(-;\mr)}{H^n(-;\mz)_\mr}) \to Hom(H^{k-1}(-;\mathbb{R}/\mathbb{Z})+\frac{\Omega^{k-1}(-)}{\Omega_0^{k-1}(-)},\frac{H^n(-;\mr)}{H^n(-;\mz)_\mr})$ are both epimorphisms.
 \end{enumerate}

\end{proposition}

%\begin{question} Is the following morphism an epimorphism in $Fun(\man^{op},\ab)$ :
%$Hom(d\Omega^{n-1}(-) \oplus Ext(H_{n-1}(-,\mz),\mz),\frac{H^n(-,\mr)}{H^n(-,\mz)_\mr}) \xrightarrow{\alpha} Ext(Hom(H_n(-,\mz)),\frac{H^n(-,\mr)}{H^n(-,\mz)_\mr})$ ? Or equivalently, is the morphism $Ext(Hom(H_n(-,\mz)),\frac{H^n(-,\mr)}{H^n(-,\mz)_\mr}) \xrightarrow{\beta} Ext(Y(-),\frac{H^n(-,\mr)}{H^n(-,\mz)_\mr})$ the zero morphism where $Y(-)=H^n(-,\mz) \underset{Hom(H_n(-;\mz),\mz)}{\times} \Omega^n_0(-)$ ?
%\end{question}
In this proposition, the hom sets are in the $Fun(\man^{op},\ab)$ category. The functor $\frac{H^n(-;\mr)}{H^n(-;\mz)_\mr}$ is considered as a subobject of the functor $H^{k-1}(-;\mathbb{R}/\mathbb{Z})+\frac{\Omega^{k-1}(-)}{\Omega_0^{k-1}(-)}$ which in in turn a subobject of $\hhat^n(-;\mrz)$.

As we have seen in the proof of proposition \ref{directdc} the corresponding question in the category $\ab$ is trivial since $\frac{H^{n-1}(M;\mr)}{H^{n-1}(M;\mz)_\mr}$ is divisible and hence injective. However, it is difficult to see whether or not the functor $\frac{H^{n-1}(-;\mr)}{H^{n-1}(-;\mz)_\mr}$ is an injective object in $Fun(\man^{op},\ab)$.
\subsection{The case of differential K-theory}\label{dk}
Complex topological K-theory too admits a differential refinement called differential K-theory. For a survey of various models of differential K-theory, see \cite{bunkeschick}. In \cite{ssdk}, Simons and Sullivan develop a model of differential K-theory for compact manifolds as the Grothendieck completion of the semigroup of 'structured vector bundles' and show that this group fits into a hexagon diagram :
\begin{proposition}
 The differential K-groups fit into the hexagon diagram 
\begin{center}
\setlength{\unitlength}{0.5cm}
\begin{picture}(24,16)
\put(5,0.5){$0$}
\put(20.5,0.5){$0$}
\put(5.5,1){\vector(1,1){1.7}}
\put(18.5,3){\vector(1,-1){1.7}}
\put(8,4){$\frac{\Omega^{odd}(M)}{\Omega_{GL}(M)}$}
\put(12,4){\vector(1,0){3.2}}
\put(17,4){$\Omega_{BGL}(M)$}
\put(13.5,4.2){\small{$d$}}
\put(2.5,6){\vector(1,1){1.7}}
\put(6.5,7){\vector(1,-1){1.7}}
\put(7.2,6.6){\small{$\zeta$}}
\put(10.5,6.5){\small{$i$}}
\put(10.3,5.7){\vector(1,1){1.7}}
\put(14.5,7){\vector(1,-1){1.7}}
\put(15.0,7){\small{$curv$}}
\put(18.5,7){\small{$s$}}
\put(18.5,6){\vector(1,1){1.7}}
\put(22.5,7){\vector(1,-1){1.7}}
\put(4,8){$H^{odd}(M;\mc)$}
\put(12,8){$\hat{K}(M)$}
\put(20,8){$H^{even}(M;\mc)$}
\put(2.5,11){\vector(1,-1){1.7}}
\put(6.5,10.5){\small{$\chi$}}
\put(6.5,9.5){\vector(1,1){1.7}}
\put(10,11){\vector(1,-1){1.7}}
\put(11,10.5){\small{$j$}}
\put(14.2,10.2){\small{$\delta$}}
\put(14.0,9.5){\vector(1,1){1.7}}
\put(18.5,11){\vector(1,-1){1.7}}
\put(19.5,10.5){\small{$r$}}
\put(22,9.5){\vector(1,1){1.7}}
\put(7,12){$K(\mc/\mz)(M)$}
\put(12.5,12){\vector(1,0){3}}
\put(16,12){$K(M)$}
\put(13.2,12.2){\small{$-B$}}
\put(5.5,15){\vector(1,-1){1.7}}
\put(18,13){\vector(1,1){1.7}}
\put(5,15){$0$}
\put(20,15){$0$}
\end{picture}
Differential K-theory hexagon diagram
\end{center}Here $\chi$ is reduction mod $\mz$, and $\zeta$ is induced by the de Rham map.
\end{proposition}
For a fuller description of the terms and maps in this diagram, see \cite{ssdk}. Throughout this subsection, i.e. in the context of differential $K-$theory, $M$ is assumed to be a compact manifold.
Simons and Sullivan ask whether the diagram determines the groups $\hat{K}(M)$ upto isomorphism compatible with the other maps in the diagram. The following proposition provides a partial answer to their question.
\begin{proposition}\label{directdk}
If $\hat{K}'(M)$ is any other abelian group together with maps $i',j',\delta',curv'$ which makes the above hexagon diagram commute, and have short exact diagonals, then there exists an isomorphism $\phi : \hat{K}(M) \to \hat{K}'(M)$ such that $\phi \circ i =i', \phi \circ j = j', \delta' \circ \phi =\delta$, and $curv' \circ \phi =curv$. 
\end{proposition}
\begin{proof}
The diagram can be redrawn as 
% https://tikzcd.yichuanshen.de/#N4Igdg9gJgpgziAXAbVABwnAlgFyxMJZABgBoBGAXVJADcBDAGwFcYkRiQBfU9TXfIRTlSxanSat2nHn2x4CRAEyjxDFm0QduvEBnmCiAZlU11UrTN36BilABYKayZu2y9-BUOSOlzjdI6crbejkb+Fm7WnoYoJvYRrlbBXsqkCWYuge42qcLpidnRBnYkpOGZAZZBHiXeZH6Vkcm1IUQiVE2uADrdAGYATvQAxsAAEgB6wEpgALTkXAAUALIA3L0AtvQ4ABYARnvAAMJcAJRcwADWMAOLvcM7WKcAvNe3vQBeMDj05zW5sWQKk6EiqIF6gxGwF6AHkNjAAOb0KYzeZLZbnaHdOGI+gAfWAAHEADJcFFzBYrP45GKlEwg8yuN6LOCnf607wiRqgyIAaTu3S2uwOxy4AHpNtt9ocAFpndl1NIVHmuXnk9FsmmKuKkbmM9iw+FI9UEgBCJI1Cra+WV+q0WA2i1NmuK1uQJltWXtjoGLwdLJdKUBKj1XvB3R222AvLJ0zAlq44hgUAR8CIoEGEA2SHIKhAOAgOYArO5M9nEOQizQCzmABylgZZnMmfOFxBKMgq9hYGplnOOVs5gDsDab7c7NYrADYugburBGD9e43y+Qh9W2+QAJyj1e1jdIHe6PvTg+IevHlc5rdnkeXsfkCebi8Zq8V-eD6e7nMzz-b78Vjen5TgBa5niW96rlWn4drOWgAFbLg+Ih-hBr5jk+OZKKBeaTuQRigS2eH2KBA6TiRkE5meuYAbhm4EZRiBEZuFHoeWsF-uuXZaMMzADLQ3CUFwQA
\begin{center}
\begin{tikzcd}
            & 0 \arrow[d]                                                             & 0 \arrow[d]                                                             & 0 \arrow[d]                              &   \\
0 \arrow[r] & \frac{H^{odd}(M;\mathbb{C})}{ker(\chi)=ker(\zeta)} \arrow[d] \arrow[r] & \frac{\Omega^{odd}(M)}{\Omega_{GL}(M)} \arrow[d, "i"] \arrow[r] & ker(s) \arrow[d] \arrow[r]               & 0 \\
0 \arrow[r] & K(\mathbb{C}/\mathbb{Z}) \arrow[d] \arrow[r, "j"]                       & \hat{K}(M) \arrow[d, "\delta"] \arrow[r, "curv"]                   & \Omega_{BGL}(M) \arrow[d] \arrow[r] & 0 \\
0 \arrow[r] & im(B) \arrow[d] \arrow[r]                                               & K(M) \arrow[d] \arrow[r]                                           & im(r)=im(s) \arrow[d] \arrow[r]          & 0 \\
            & 0                                                                       & 0                                                                       & 0                                        &  
\end{tikzcd}
\end{center}
The proposition follows from corollary \ref{corpinj} by noting that $P=\frac{H^{odd}(M;\mc)}{ker(\chi)=ker(\zeta)}$ is divisible and hence an injective object in $\ab$.
\end{proof}
This is weaker than the claim of uniqueness of the functor $\hat{K}$. Consider the functor category $Fun(\man^{op}_{cpt},\ab)$. From \ref{finaluniqueness}, we note the following 
\begin{proposition}\label{dkconditions}
 The following are equivalent : 
 \begin{enumerate}
  \item If $\hat{K}'(-)$ is another functor from $\man^{op}_{cpt}$ to $\ab$ together with natural transformations $i',j',\delta',curv'$ which fit in the hexagon diagram with exact diagonals, then there is a natural equivalence $\phi : \hat{K}(-) \to \hat{K}'(-)$ such that $\phi \circ i=i', \phi \circ j=j', \delta' \circ \phi =\delta$, and $curv' \circ \phi =curv$.
  \item The natural transformation $Hom(R \oplus S, P) \xrightarrow{\alpha} Ext^1(Q,P)$ is an epimorphism in $Fun(\man^{op}_{cpt},\ab)$, and every natural transformation from $E+H$ to $P$ which vanishes on $P$ extends to a natural transformation on $X$ %  $Hom(X,P) \to Hom(E+H,P)$ are epimorphisms
  where $P=\frac{H^{odd}(-;\mc)}{ker(\chi)=ker(\zeta)}, E= \frac{\Omega^{odd}(-)}{\Omega_{GL}(-)}, H=K(\mathbb{C}/\mz)(-), R = ker(s), S= im(B)$, and $X=\hat{K}(-) \in Fun(\man^{op}_{cpt},\ab)$.
 \end{enumerate}
\end{proposition}
Like in proposition \ref{dcconditions}, we consider $P$ as a subobject of $E+H$, which is in turn a subobject of $X$.\\
A stronger and more general result showing the uniqueness of the differential version of exotic cohomology theories with integration (and hence differential K-theory functor, in particular) has been proved in \cite{bunke} by a different method.
\section*{Acknowledgments}
I wish to thank my supervisor Dr. Rishikesh Vaidya for discussions and support. I am thankful to Jitendra Rathore for valuable discussions. I am grateful to the Council of Scientific \& Industrial Research for financial support under the CSIR-SRF(NET) scheme.
\bibliography{refs}
Ishan Mata, \textsc{Department of Physics, Birla Institute of Technology and Science - Pilani, Pilani campus, Pilani, Rajasthan, India. PIN : 333031}\\
\textit{Email} : ishanmata@gmail.com 
\end{document}